\documentclass[10pt]{amsart}
\usepackage[most]{tcolorbox}
\usepackage{amssymb, amsmath}
\usepackage{graphics,xcolor}
\usepackage{latexsym,amsmath,amssymb,amsthm,amsfonts,mathrsfs}
\usepackage{amscd}
\usepackage[arrow, matrix, curve]{xy}
\usepackage{tikz}
\usetikzlibrary{positioning, 
                quotes}

\DeclareMathOperator{\aut}{Aut}

\DeclareMathOperator{\gl}{GL}

\DeclareMathOperator{\jac}{Jac} 
\DeclareMathOperator{\pic}{Pic}
\DeclareMathOperator{\nm}{Nm}

\DeclareMathOperator{\im}{Im}

\DeclareMathOperator{\py}{Prym}

\DeclareMathOperator{\dv}{Div}
\DeclareMathOperator{\vol}{Vol}
\DeclareMathOperator{\gram}{Gram}

\theoremstyle{plain}
\newtheorem{thm}{Theorem}[section]

\newtheorem{lemma}[thm]{Lemma}

\newtheorem{proposition}[thm]{Proposition}

\theoremstyle{definition}

\newtheorem{remark}[thm]{Remark}

\newtheorem{definition}[thm]{Definition}

\newtheorem{example}[thm]{Example}

\numberwithin{equation}{thm}

\newcommand{\sE}{{\mathcal E}}

\newcommand{\N}{{\mathbb N}}

\newcommand{\R}{{\mathbb R}}

\newcommand{\Z}{{\mathbb Z}}

\title[]{tropical covers, tropical abelian varieties and Prym varieties}
\author{Abolfazl Mohajer}
\address{}
\email{abmohajer83@gmail.com}

\begin{document}
\begin{abstract}
We define and investigate the tropical Prym varieties  associated to unramified Galois cyclic covers of tropical curves (or equivalently metric graphs) $\tilde{\Gamma}\to \Gamma$. Our approach here is to study the tropical Prym varieties using group actions on tropical abelian varieties induced by the cyclic Galois group of the cover of tropical curves. We also define and consider the Abel-Prym map for tropical cyclic covers extending that for double covers. As a special case we consider free $\Z_3$-covers of tropical curves and their associated tropical Prym variety and compute its volume generalizing the case of double covers. 
\end{abstract}
	
\maketitle

\section{introduction}
It is well-known that if $X$ is a Riemann surface and $B\subset X$ is a finite set, then there is a unique cover $Y\to X$ ramified exactly over $B$ with prescribed monodromy. GIven a finite group $G$, the classical theory of Galois covers of Riemann surfaces sayy that Galois cover of $X$, covers with deck transformation group given by $G$, are classified by giving a monodromy represenation $\pi_1(X,x_0)\to G$ for some $x_0\in X\setminus B$.  There is an analogy to this in tropical geometry discussed and reinforced by many authors, see for example \cite{LUZ24}.\par The tropical counterpart of a non-constant holomorphic map of Riemann surfaces is a finite harmonic map of metric graphs (equivalently tropical curves).  This need not be a topological cover, indeed harmonic maps can dilate the edges of the graph, namely restricted to an edge $e$ identified with an interval $[0,a]$, it is given by a map $x\mapsto dx$. Associated to a cover of tropical curves $Y\to X$, there is a Prym variety analogous to the Prym variety obtained by covers of Riemann surfaces. The tropical Prym variety is an abelian tropical variety. The classical Prym variety is an isogeny factor of the Jacobian variety of $Y$ and has many properties that can be controlled through the covering map. Our aim in this paper is to establish similar properties for the Prym variety associated to tropical curves or equivalently metric graphs, the so-called tropical Prym variety of the tropical cover. \par 
In the next section, we introduce harmonic morphisms of metric graphs and explore properties of harmonic morphisms and consider Galois covers of tropical curves (metric graphs). \par In Section 3, we consider tropical abelian varieties and group actions on them. This section paves the way for Section 4 in which we introduce the tropical Prym variety by using the group action of the Galois group of tropical Galois cover $\varphi:\widetilde{\Gamma}\to\Gamma$ on the tropical Jacobian $\jac(\widetilde{\Gamma})$. We then explore various properties of Prym varieties using this defition. We also define the so-called Abel-Prym map $\widetilde{\Gamma}\to P$  for $\Z_n$-Galois cover of tropical covers. This generalizes the Abel-Prym map fro double covers considered in \cite{LZ22}. \par In particular we consider triple covers $\varphi:\widetilde{\Gamma}\to\Gamma$, that is, $\Z_3$-Galois covers of tropical curves or metric graphs. In this case we explicitely compute the volume of Prym variety and extend the formula in Proposition 3.6 in \cite{LZ22} for the volume of tropical Prym variety associated to double covers of metric graphs.

\section{Harmonic morphisms of metric graphs, Galois covers and their properties}
Let $\Gamma$ be a finite graph. We denote the set of vertices and edges of the graph respectively by $V(\Gamma)$ and $E(\Gamma)$. The genus of the graph is defined by $g(\Gamma)=|E(\Gamma)|-|V(\Gamma)|+1$. A \emph{metric graph} is the compact metric space obtained from a finite graph $V(\Gamma)$ by assigning positive lengths $\ell:E(\Gamma)\to\R_{>0}$ to its edges and identifying each edge $e\in E(\Gamma)$ with a closed interval of length $\ell(e)$. 
\begin{definition}\label{morph}
A continuous map $\varphi:\widetilde{\Gamma}\to\Gamma$ of metric graphs is called a morphism if there exist vertex sets $\widetilde{V}\subseteq\widetilde{\Gamma}$ and $V\subseteq\Gamma$ such that $\varphi(\widetilde{V})\subseteq V, \varphi^{-1}(E(\Gamma))\subseteq E(\widetilde{\Gamma})$ and the restriction of $\varphi$ to any edge $\widetilde{e}$ of $\widetilde{\Gamma}$ is a dilation by some factor $d_{\widetilde{e}}(\varphi)\in\N_0$. A morphism $\varphi$ is called finite if $d_{\widetilde{e}}(\varphi)>0$ for every edge $\widetilde{e}$ of $\widetilde{\Gamma}$. 
\end{definition}

Let $\varphi:\widetilde{\Gamma}\to\Gamma$ be finite morphism of metric graphs. Let $\widetilde{x}\in\widetilde{\Gamma}$ and set $x=\varphi(\widetilde{x})$.
Further let $v\in T_{x}(\Gamma)$ be a tangent vector at $x$. We define
\begin{equation} \label{loc.deg}
d_{\widetilde{x},v}(\varphi)=\displaystyle\sum_{\widetilde{v}\in\varphi^{-1}(v)\cap T_{\widetilde{x}}(\widetilde{\Gamma})}d_{\widetilde{v}}(\varphi)
\end{equation}

\begin{definition}\label{harmonic}
A finite morphism $\varphi:\widetilde{\Gamma}\to\Gamma$ of metric graphs is called harmonic at $\widetilde{x}\in\widetilde{\Gamma}$ if the quantity $d_{\widetilde{x},v}(\varphi)$ defined in \ref{loc.deg} is independent of the tangent vector $v\in T_{x}(\Gamma)$, in which case it will be denoted simply by
$d_{\widetilde{x}}(\varphi)$.  The number $d_{\widetilde{x}}(\varphi)$ is called the degree of the harmonic morphism at $\widetilde{x}$. The morphism is called harmonic if it is surjective and harmonic at every point $\widetilde{x}\in\widetilde{\Gamma}$.  If this is the case, then $\deg(\varphi)=\displaystyle\sum_{\widetilde{x}\in\varphi^{-1}(x)}d_{\widetilde{x}}(\varphi)$ is independent of $x$, and is called the degree of the harmonic morphism $\varphi$. An unramified harmonic morphism is called \emph{free} if it is of local degree $1$ over each point $x\in\Gamma$. 
\end{definition}
A free harmonic morphism is a covering space in the topological sense  that preserves adjacency meaning that it is an isomorphism in the neighbourhood of every vertex of $\widetilde{\Gamma}$. In other words, for any pair of vertices $\widetilde{v}$ and $v$ with $\varphi(\widetilde{v})=v$  and for each edge $e\in E(\Gamma)$ attached to $v$ there is a unique edge $\widetilde{e}\in E(\widetilde{\Gamma})$ attached to $\widetilde{v}$ that maps to $e$.  
We now define a harmonic Galois cover of a graph as in \cite{LUZ24}, Definition 1.4 for $G$ a finite group. 
\begin{definition}\label{tcover}
A harmonic map $\varphi:\widetilde{\Gamma}\to\Gamma$ of metric graphs is called a $G$-Galois cover if there is a $G$-action on $\widetilde{\Gamma}$ (see \cite{LUZ24}, Definition 1.2) such that the harmonic morphism $\varphi$ is $G$-invariant and $G$ acts transitively on each fiber $\varphi^{-1}(x)$.  
\end{definition}
In this article we are interested in the case where $G=\Z_n$ is a cyclic group in which case we speak of a cyclic $\Z_n$ -cover of metric graphs. 
Given an unramified $\Z_n$-cover $\varphi:\widetilde{\Gamma}\to\Gamma$ we define the dilation cycle as follows:
\begin{definition}\label{dilcyc}
Let $\varphi:\widetilde{\Gamma}\to\Gamma$ be a harmonic $\Z_n$-cover. Then
\[\gamma(\varphi)=\{x\in\Gamma| \exists\widetilde{x}\in\varphi^{-1}(x), \widetilde{v}\in T_{\widetilde{x}}(\widetilde{\Gamma}) \text{ such that }d_{\widetilde{v}}(\varphi)>1\}\]
More generally we define
\[\gamma_u(\varphi)=\{x\in\Gamma| \exists\widetilde{x}\in\varphi^{-1}(x), \widetilde{v}\in T_{\widetilde{x}}(\widetilde{\Gamma}) \text{ such that }d_{\widetilde{v}}(\varphi)=u\}\]
for every $1\leq u\leq n$. 
\end{definition}
\begin{definition}\label{genus}
An augmented metric graph is a metric graph $\Gamma$, together with a function $g:\Gamma\to\N_0$ called the genus function, such that $g(x)=0$ for all but finitely many $x\in\Gamma$. If the genus function $g$ is identically zero, then the graph $\Gamma$ is called unaugmented. 
\end{definition}
In what follows we consider the unramified covers $\varphi:\widetilde{\Gamma}\to\Gamma$. Let the genuera of the graphs be $\widetilde{g}$ and $g$ respectively. This morphism is called unramified if $\varphi^*(K_{\Gamma})=K_{\widetilde{\Gamma}}$. Explicitely, given a point $\widetilde{x}\in\widetilde{\Gamma}$, the ramification of $\varphi$ at $\widetilde{x}$ is
\[R_{\widetilde{x}}= (2-2\widetilde{g}(\widetilde{x}))-d_{\widetilde{x}}(\varphi)(2-2g(\varphi(\widetilde{x})))-\displaystyle\sum_{\widetilde{v}\in T_{\widetilde{x}}(\widetilde{\Gamma})}(d_{\widetilde{v}}(\varphi)-1),\]
in other words $R=K_{\widetilde{\Gamma}}-\varphi^*(K_{\Gamma})$ is the ramification divisor. The morphism $\varphi$ is unramified if $R=0$, or, equivalently, if $R_{\widetilde{x}}=0$ for all $\widetilde{x}\in\widetilde{\Gamma}$. 
\begin{lemma}\label{dilcyc1}
Let $\varphi:\widetilde{\Gamma}\to\Gamma$ be an unramified $\Z_n$-cover of metris graphs with $\Gamma$ unaugmented, then
\begin{enumerate}
\item If $x\notin \gamma(\varphi),$ then $\varphi^{-1}(x)$ consists of at most $n$ points of genus 0.
\item If $x\in \gamma(\varphi),$ then $\varphi^{-1}(x)$ consists of points of genus $\geq \frac{1}{2}\deg_{\gamma(\varphi)}-(d_{\widetilde{x}}-1)$.
\end{enumerate}
Here $\deg_{\gamma(\varphi)}$ denotes the number of tangent vectors at $x$ contained in $\gamma(\varphi)$.
\end{lemma}
\begin{proof}
By the fact that the genus function of $\Gamma$ is trivial, using the ramification formula we obtain
\[2g(\widetilde{x})+2(s-1)=\displaystyle\sum_{\widetilde{v}\in T_{\widetilde{x}}(\widetilde{\Gamma})}(d_{\widetilde{v}}(\varphi)-1),\] 
where $s=d_{\widetilde{x}}$. Therefore if $d_{\widetilde{x}}>1$ and $x\notin\gamma(\varphi)$, then the above equality implies that $g(\widetilde{x})<0$ which is not possible. Therefore if $d_{\widetilde{x}}>1$, then $\varphi(\widetilde{x})\in \gamma(\varphi)$. On the other hand, if $d_{\widetilde{x}}=1$, then the above equality implies that $g(\widetilde{x})=0$. 
\end{proof}

One of the most useful properties of Galois covers of algebraic curves is that in such covers, the ramification indices of ramification points above a given branch point are equal. The next definition gives such condition for covers of metric graphs.

\begin{definition}\label{strong}
An unramified $\Z_n$-cover of metric graphs $\varphi:\widetilde{\Gamma}\to\Gamma$ is called strongly unramified if for every $x\in\Gamma$ and every $\widetilde{x}\in\varphi^{-1}(x)$, the degree $d_{\widetilde{x}}$ depends only on $x$. In other words, if $\widetilde{x}_1, \widetilde{x}_2\in\varphi^{-1}(x)$, then $d_{\widetilde{x}_1}=d_{\widetilde{x}_2}$. We therefore denote this degree by $d_x$. If $\varphi$ is strongly unramified, then $\deg(\varphi)$ is a multiple of $d_{x}$. In particular, if $\deg(\varphi)=p$ is a prime number, then each $d_{x}$ is equal to 1 or $p$. 
\end{definition}

\begin{remark}\label{double}
Note that by Lemma 5.4 of \cite{JL}, any harmonic double cover is  strongly unramified. 
\end{remark}

\begin{lemma}\label{dilcyc strong}
Let $\varphi:\widetilde{\Gamma}\to\Gamma$ be a strongly unramified $\Z_n$-cover of metris graphs with $\Gamma$ unaugmented, then
\begin{enumerate}
\item If $x\notin \gamma(\varphi),$ then $\varphi^{-1}(x)$ consists of precisely $n$ points of genus 0.
\item If $x\in \gamma(\varphi),$ then $\varphi^{-1}(x)$ consists of $\frac{n}{d_{x}}$ points of genus \par  $\geq\frac{1}{2}\deg_{\gamma(\varphi)}-(d_{x}-1)$.
\end{enumerate}
\end{lemma}

\section{tropical abelian varieties with group action}
In this section we consider tropical abelian varieties and group actions on them. This will be used in the next section to introduce the tropical Prym variety of tropical Galois covers $\varphi:\widetilde{\Gamma}\to\Gamma$ by considering the group action on the tropical Jacobian $\jac(\widetilde{\Gamma})$. 
\begin{definition}\label{tropabvar}
A real torus is a quotient $\R^n/\Lambda$, where $\Lambda\subset\R^n$ is a lattice in $\R^n$ of rank $n$. A polarization $Q$ on a reall torus is a $\Lambda\to (\Z^n)^*$. A reall torus is called a tropical abelian variety if $Q$ induces a positive definite and symmetric bilinear form. The polarization is called principal if it induces an isomorphism from $\Lambda$ to $(\Z^n)^*$. We sometimes denote a principally polarized tropical abelian variety (pptav) as a pair $(\R^n/\Lambda, Q)$.
\end{definition}

\begin{definition}\label{isomtrop}
 Two tropical abelian varieties $(\R^n/\Lambda, Q)$ and $(\R^n/\Lambda^{\prime}, Q^{\prime})$ are isomorphic if there exists $h\in\gl(n,\R)$ such that $h(\Lambda)=\Lambda^{\prime}$ and $hQh^t=Q^{\prime}$. An automorphism of a pptav $A=(\R^n/\Lambda, Q)$ is given by an element $h\in\gl(n,\R)$ such that $h(\Lambda)=\Lambda$ and $hQh^t=Q$, i.e., an element $h$ which fixes the lattice and the positive definite quadratic form $Q$. The set of all automorphisms of a pptav forms a group under composition (equivalently matrix multiplication) which we denote by $\aut(A)$.
\end{definition}

\begin{remark}\label{tropabvaresp}
Every tropical abelian variety $(\R^n/\Lambda, Q)$ can be written in the form $(\R^n/\Z^n, Q^{\prime})$, where we set $Q^{\prime}=hQh^t$, and $h\in\gl(n,\R)$ is such that $h(\Lambda)=\Z^g$. From now on we therefore consider the tropical abelian varieties in the form $(\R^n/\Z^n, Q)$, where $Q$ is uniquely determined up to arithmetic equivalence i.e., up to the identification $Q\sim hQh^t$ for some $h\in\gl(n,\Z)$. 
\end{remark}

\begin{definition}\label{group action}
Let $A=(\R^g/\Z^g, Q)$ be a pptav and $G$ a finite gorpup. We say that $G$ acts on $A$ if  $G$ admits a group monomorphism $G\hookrightarrow \aut(A)$. 
\end{definition}

\section{tropical Prym variety and Abel-Pram map associated to cyclic covers}
{\bf Construction of tropical Jacobian variety.} The construction of the Jacobian of a metric graph as a pptav can be taken from \cite{LU21} or  \cite{BF11} . We briefly recall this construction here as follows: Let $A$ be either $\Z$ or $\R$. The group of $A$-valued differential forms $\Omega(\Gamma,A)$ on $\Gamma$ is a subgroup of the free $A$-module with basis $\{de| e\in E(\Gamma)\}$ defined as
\[\Omega(\Gamma,A)=\{\displaystyle\sum_{e\in E(\Gamma)}\omega_e de|\sum_{t(e)=v}\omega_e=\sum_{s(e)=v}\omega_e \forall v\in V(\Gamma)\}\]
This group can be identified with the first homology $H_1(\Gamma, A)$ of the graph $\Gamma$ with coefficeints in $A$. By lemma 2.1 in \cite{BF11}, there is a perfect pairing on $H_1(\Gamma, A)\times \Omega(\Gamma,A)$ which is induced from an integration pairing 
\[[\cdot,\cdot]: C_1(\Gamma, A)\times \Omega(\Gamma,A)\to\R\]
by
\[ [\gamma,\omega]= \displaystyle\int_{\gamma}\omega=\sum_{e\in E(\Gamma)}\gamma_e\omega_e l(e), \gamma=\sum_{e\in E(\Gamma)}\gamma_e e, \omega=\sum_{e\in E(\Gamma)}\omega_e de\]
Here $C_1(\Gamma, A)=A^{E(\Gamma)}$ and $C_0(\Gamma, A)=A^{V(\Gamma)}$ and $H_1(\Gamma, A)$ is just the kernel of the map 
\[d:C_1(\Gamma, A)\to C_0(\Gamma, A)\]
\[e\mapsto t(e)-s(e)\]

The Jacobian variety of $\Gamma$ is then defined as
\[\jac(\Gamma)=\Omega(\Gamma)^*/H_1(\Gamma, \Z),\]
Let $\varphi:\widetilde{\Gamma}\to\Gamma$ be an unramfied cyclic $\Z_n$ -cover of metric graphs.  By Definition \ref{tcover}, this cover corresponds to an automorphism $\tau:\widetilde{\Gamma}\to\widetilde{\Gamma}$ of order $n$ which also induces an automorphism of order $n$ of the Jacobian
\begin{equation}\label{}
\tau:\jac(\widetilde{\Gamma})\to \jac(\widetilde{\Gamma})
\end{equation}
which we again denote by $\tau$. Indeed we consider $\tau$ to be the generator of the group $G=\Z_n$ that acts on $\jac(\widetilde{\Gamma})$. Note that we have also the norm map $\nm_{\varphi}$. \par Using the norm map above we can define the Prym variety associated to an unramified cyclic cover.

\begin{align}\label{}
\nm_{\varphi}:\dv(\widetilde{\Gamma})\to \dv(\Gamma)\nonumber\\
\sum a_i p_i\mapsto \sum a_i \varphi(p_i)
\end{align}
We denote the induced norm map $\nm_{\varphi}:\pic^0(\widetilde{\Gamma})\to\pic^0(\Gamma)$ also by $\nm_{\varphi}$.
Note that we have also the pull-back map $\varphi^*:\jac(\Gamma)\to \jac(\widetilde{\Gamma})$ which is induced by the pull-back map of divisors $\varphi^*:\dv(\Gamma)\to \dv(\widetilde{\Gamma})$.

\begin{definition}\label{prym def}
Let $\varphi:\widetilde{\Gamma}\to\Gamma$ be an unramfied cyclic cover of degree $n$ of tropical curves (or metric graphs). 
We define the \emph{Prym variety} $\py(\widetilde{\Gamma}/\Gamma)$ of $\varphi$ as the connected component of the identity of $\ker\nm_{\varphi}$, in other words,
\begin{equation}\label{prymdef}
\py(\widetilde{\Gamma}/\Gamma):=(\ker\nm_{\varphi})^{\circ}
\end{equation}
\end{definition}
The Prym variety $\py(\widetilde{\Gamma}/\Gamma)$ has the structure of a principally polarized tropical abelian variety (pptav). Let us describe this structure here. Let us denote $\widetilde{\Lambda}=\Omega(\widetilde{\Gamma},\Z), \widetilde{\Lambda^{\prime}}=H_1(\widetilde{\Gamma},\Z), \Lambda=\Omega(\Gamma,\Z)$ and $\Lambda^{\prime}=H_1(\Gamma,\Z)$. Next we consider the pushforward and pull-back maps $\varphi_*:H_1(\widetilde{\Gamma},\Z)\to H_1(\Gamma,\Z)$ and $\varphi^*:\Omega(\Gamma,\Z)\to \Omega(\widetilde{\Gamma},\Z)$ given by
\[\varphi_*(\widetilde{e})=\varphi(\widetilde{e}),\hspace{1cm} \varphi^*(de)=d\widetilde{e}^1+d\widetilde{e}^2+d\widetilde{e}^3,\]
for $\widetilde{e}\in E(\widetilde{\Gamma})$ and $e\in E(\Gamma)$ and the edges $\widetilde{e}^i$ are the preimages of $e$ in $\widetilde{\Gamma}$ and then extending by linearity. The norm map $\nm_{\varphi}:\pic^0(\widetilde{\Gamma})\to\pic^0(\Gamma)$ is indeed induced by these maps, in fact we have
\[\py(\widetilde{\Gamma}/\Gamma)=\frac{\ker({\varphi^*}^{\vee}:\Omega(\widetilde{\Gamma})\to\Omega(\Gamma))}{\ker(\varphi_*:H_1(\widetilde{\Gamma},\Z)\to H_1(\Gamma,\Z))}.\]
By Theorem 2.2.7 in \cite{LU21} (which is proven for general finite covers and general Prym varieties as here), this variety has an induced polarization from $\jac(\widetilde{\Gamma})$ and hence is a pptav with the inner product $(.,.)_P$ satisfying $(.,.)_P=\frac{1}{n}(.,.)_{\widetilde{\Gamma}}$. \par

Let us explore the tropical Prym variety for cyclic covers further. For the sake of simplicity, let us denote $J=\jac(\Gamma)$ and $\widetilde{J}=\jac(\widetilde{\Gamma})$. We introduce another morphism $\nm_G:\widetilde{J}\to \widetilde{J}$ as
\[\nm_G:=\sum_{g\in G}g=1+\tau+\cdots+\tau^{n-1}.\]
With these notation we have

\begin{lemma}\label{}
\begin{enumerate}
\item $\nm_{\varphi}\circ\varphi^*=n\cdotp 1_{J}$.
\item $\varphi^*\circ\nm_{\varphi}=\nm_G$.
\item $\varphi^*(J)={(\widetilde{J}^G)}^{\circ}$ and $\py(\widetilde{\Gamma}/\Gamma)\cap {\widetilde{J}}^G\subseteq P[n]$. 

\end{enumerate}
\end{lemma}

We have the following

\begin{proposition}\label{prym cyc}
Let $\varphi:\widetilde{\Gamma}\to\Gamma$ be an unramfied cyclic cover of degree $n$ of tropical curves (or metric graphs). 
Let $\tau$ be the corresponding automorphism of $\jac(\widetilde{\Gamma})$ as above. We have
\begin{equation}\label{prymdef}
\py(\widetilde{\Gamma}/\Gamma)=\ker(\nm_G)^0=\im(1-\tau)
\end{equation}
\end{proposition}
Note that the morphism $\varphi$ induces also a morphism $\mu:\jac(\widetilde{\Gamma})\to \jac(\Gamma)$ which is in turn induced by the surjective map $\sigma_*:\Omega(\widetilde{\Gamma})^*\to \Omega(\Gamma)^*$, the dual of the natural pull-back map of harmonic 1-forms $\sigma^*:\Omega(\Gamma)\to \Omega(\widetilde{\Gamma})$ which is injective. Let $\Omega(\widetilde{\Gamma}/\Gamma)^*$ be the kernel of $\sigma_*$. By quotienting this out by $H_1(\widetilde{\Gamma}, \Z)$ and $H_1(\Gamma, \Z)$ be obtain the map $\mu$. We may therefore identify $\ker\mu$ with $\sigma^{-1}_* H_1(\Gamma, \Z)$. Since the composition $\sigma_*\sigma^*$ is just multiplication by $n$, we obtain

\begin{equation}\label{}
\sigma^{-1}_* H_1(\Gamma, \Z)=\Omega(\widetilde{\Gamma}/\Gamma)^*+\frac{1}{n}\sigma^* H_1(\Gamma, \Z)
\end{equation}
By taking quotient, we obtain that $\ker\varphi_*$ is the translation of the torus 
\begin{equation}\label{}
\Omega(\widetilde{\Gamma}/\Gamma)^*/(H_1(\widetilde{\Gamma}, \Z)\cap \Omega(\widetilde{\Gamma}/\Gamma)^*)\subset \jac(\widetilde{\Gamma})
\end{equation}
 Then the kernel of $\mu:\jac(\widetilde{\Gamma})\to\jac(\Gamma)$ consists of a finite union of translations of the abelian subvariety $\im(1-\tau)$ by elements of $\sigma^*\jac_n(\Gamma)$. \\

If a group $G$ acts on a metric graph $\Gamma$, then $G$ acts also on the homology group $H_1(\Gamma, \R)$ and therefore on the Jacobian $\jac(\Gamma)$. In this section, we assume that $G$ is a finite abelian abelian group. We then have
\begin{gather}
H_1(\Gamma, \R)^+=H_1(\Gamma, \R)^{G}\\
H_1(\Gamma, \R)^-=H_1(\Gamma, \R)/H_1(\Gamma, \R)^+
\end{gather}
we have $H_1(\Gamma, \R)=H_1(\Gamma, \R)^+\oplus H_1(\Gamma, \R)^-$. In the same way, we can also define $H_1(\Gamma, \Z)^+$ and $H_1(\Gamma, \Z)^-$.  The group $G$ also acts on the space of 1-forms $\Omega(\Gamma)$ and we have the eingenspaces
\begin{gather}
\Omega(\Gamma)^+=\Omega(\Gamma)^{G}\\
\Omega(\Gamma)^-=\Omega(\Gamma)/\Omega(\Gamma)^{+}
\end{gather}

The group $H_1(\Gamma, \Z)$ carries an intesection form which we denote by $(.,.)_{\Gamma}$. Similarly, the  intesection form on $H_1(\widetilde{\Gamma}, \Z)$ is denoted by $(.,.)_{\widetilde{\Gamma}}$. Let $(.,.)_P$ be the intersection pairing on $\ker p_*$ corresponding to the principal polarization on $\py(\widetilde{\Gamma}/\Gamma)$. 

Now if $\varphi:\widetilde{\Gamma}\to\Gamma$ is an unramfied cyclic cover of degree $n$, then as explained earlier, there is a $\Z_n$-action on $\widetilde{\Gamma}$ and on $\jac(\widetilde{\Gamma})$ generated by $\tau$ introduced above. Proposition \ref{prym cyc} and the previous discussions then give
\begin{proposition}\label{prym comp}
Let $\varphi:\widetilde{\Gamma}\to\Gamma$ be an unramfied cyclic cover of degree $n$ of tropical curves (or metric graphs). Then it holds
\begin{equation}\label{Prym quot}
\py(\widetilde{\Gamma}/\Gamma)={\Omega(\widetilde{\Gamma})^-}^*/H_1(\widetilde{\Gamma}, \Z)^-
\end{equation}
\end{proposition}

The following proposition describes the tropical Prym variety in terms of isogenies of real abelian varieties . 

\begin{proposition}\label{prym comp}
Let $\varphi:\widetilde{\Gamma}\to\Gamma$ be an unramfied cyclic cover of degree $n$ of tropical curves. Then $\py(\widetilde{\Gamma}/\Gamma)$ is complement of the abelian subvariety $\mu^*\jac(\Gamma)$ in $\jac(\widetilde{\Gamma})$. In other words we have the following isogeny of real abelian varieties 
\begin{equation}\label{isog prym}
\jac(\widetilde{\Gamma})\sim \py(\widetilde{\Gamma}/\Gamma)\times \mu^*\jac(\Gamma)
\end{equation}
\end{proposition}

Recall from \cite{MZ08} that for every metric graph $\Gamma$ we have the tropical \emph{Abel-Jacobi} map $\alpha:\Gamma\to\jac(\Gamma)$. Using this and our defintion of the Prym variety of a covering  $\varphi:\widetilde{\Gamma}\to\Gamma$ we can now define

\begin{definition}\label{abel-prym}
Let $\varphi:\widetilde{\Gamma}\to\Gamma$ be an unramfied cyclic cover of degree $n$ of metric graphs. Let $P=\py(\widetilde{\Gamma}/\Gamma)$ and let $\tau$ be as in the beginning of this section. The tropical \emph{Abel-Prym map} of $\varphi$ is defined as the composition
\begin{equation}\label{compos AP}
\Psi: \widetilde{\Gamma}\xrightarrow{\alpha}\jac(\widetilde{\Gamma})\xrightarrow{1- \tau} P
\end{equation}
\end{definition}


\subsection{Tropical triple covers.}

Let $\Gamma$ be a graph of genus $g$. Fix a spanning tree $T\subset\Gamma$ and a seubset $S\subseteq\{e_0,\ldots, e_{g-1}\}$ of the edges in the complement $E(\Gamma)\setminus E(T)$. Let $T^1,T^2$ and $T^3$ be three copies of $T$and for a vertex $v\in V(T)=V(\Gamma)$ denote $v^1, v^2, v^3$ be the corresponding vertices in $T^{i}$ for $i=1,2,3.$ We define
\[\widetilde{\Gamma}=T^1\cup T^2\cup T^3\cup\{(e_i)^{1,2,3}\}.\]
The map $p:\widetilde{\Gamma}\to\Gamma$ sends $T^{1,2,3}$ isomorphically to $T$ and $(e_i)^{1,2,3}$ to $e_i$. For each $e_i\in S,$ the three edges $(e_i)^{1,2,3}$ have one vertex on $T^1, T^2, $ and $T^2, T^3$ and $T^1, T^3$. Note that the three graphs above the loop resulting from the simple loop associated to $e_i$ are isomorphic as triple covers. If $e_i\notin S,$ then the two  edges of $(e_i)^{1,2,3}$ are on the same tree $T^1, T^2$ or $T^3$. We make the convention that $e_0^1$ has its source vertex on $T^2$ and target vertex on $T^1$, $e_0^2$ has its source vertex on $T^3$ and target vertex on $T^2$ and $e_0^3$ has its source vertex on $T^1$ and target vertex on $T^3$. A spanning tree for the graph $\widetilde{\Gamma}$ is $\widetilde{T}=T^1\cup T^2\cup T^3\cup\{e_0^1\}\cup\{e_0^2\}$.\\

Let us construct a basis for $H^1(\Gamma, \Z)$ and $H^1(\widetilde{\Gamma}, \Z)$. Let $\gamma_i\in H^1(\Gamma, \Z)$  for $i=0,\ldots, g-1$ denote the unique cycle of $T\cup\{e_i\}$ such that $\langle \gamma_i, e_i\rangle=1$. Similarly, let $\widetilde{\gamma}_0\in H^1(\widetilde{\Gamma}, \Z)$  and 
$\widetilde{\gamma}_i^j\in H^1(\widetilde{\Gamma}, \Z)$ be the unique cycle respectively of $\widetilde{\Gamma}\cup\{e_0^3\}$ and $\widetilde{\Gamma}\cup\{e_i^j\}, j=1,2,3$. The cycle $\widetilde{\gamma}_0$ starts at $v^1$ and then proceeds to $w^3$ via $+e_0^3$ and then to $v^3$ via a unique path on $T^3$ then to $w^2$ via $+e_0^2$ and then $v^2$ via a unique path on $T^2$ and then to $w^1$ via $+e_0^1$ and then again back to $v^1$ via a unique path on $T^1$.  In other words,
\[\widetilde{\gamma}_0=e_0^1+e_0^2+e_0^3+\sE(\widetilde{T}), \tau_*(\widetilde{\gamma}_0)=e_0^1+e_0^2+e_0^3+\sE^{\prime}(\widetilde{T}),
p_*(\widetilde{\gamma}_0)=3e_0+\sE(T)\]
where $\sE(\widetilde{T}), \sE^{\prime}(\widetilde{T})$ and $\sE(T)$ denote a set of edges on $\widetilde{T}$ and $T$ respectively. Now computing these in the bases given above, we get
\[\tau_*(\widetilde{\gamma}_0)=\widetilde{\gamma}_0, p_*(\widetilde{\gamma}_0)=3\gamma_0\]

Let us compute the map $p_*:H^1(\widetilde{\Gamma}, \Z)\to H^1(\Gamma, \Z)$ and the map $\tau_*:H^1(\widetilde{\Gamma}, \Z)\to H^1(\widetilde{\Gamma}, \Z)$ on the basis $\widetilde{B}$.

Consider the cycle $\widetilde{\gamma}_i^1$ for $e_i\in S\setminus\{e_0\}$. Let us define the index 
\[\sigma_i=\begin{cases}
+1 \quad\text{ if } s(e_i^1)=s(e_i)^2 (\text{hence } t(e_i^1)=t(e_i)^1)\\
 -1  \quad\text{ if } s(e_i^1)=s(e_i)^1 (\text{hence } t(e_i^1)=t(e_i)^2)
\end{cases}\]
Note that, by our convention, it automatically implies $s(e_i^2)=s(e_i)^3$ (so $t(e_i^2)=t(e_i)^2$) and $s(e_i^3)=s(e_i)^1$ (so $t(e_i^3)=t(e_i)^3$) in the case $\sigma_i=+1$ and $s(e_i^2)=s(e_i)^2$ (so $t(e_i^2)=t(e_i)^3$) and $s(e_i^3)=s(e_i)^3$ (so $t(e_i^3)=t(e_i)^1$) in the case $\sigma_i=-1$.

If $\sigma_i=+1$, then the cycle $\widetilde{\gamma}_i^1$ starts at $s(e_i^1)=s(e_i)^2$ and then moves to $t(e_i^1)=t(e_i)^1$ on $T_1$ via $e_i^1$. Since this path passes from $T^2$ to $T^1$, it must contain $-e_0^1$. If $\sigma_i=-1$, then the cycle $\widetilde{\gamma}_i^1$ starts at $s(e_i^1)=s(e_i)^1$ and then moves to $t(e_i^1)=t(e_i)^2$ on $T^2$ via $e_i^1$. Since this path passes from $T^1$ to $T^2$, it must contain $-e_0^1$. Similar considerations hold for $\widetilde{\gamma}_i^1$ and $\widetilde{\gamma}_i^2$. Hence we get
\[\widetilde{\gamma}_i^1=e_i^1-\sigma_i e_0^1+E(\widetilde{T}), \tau_*(\widetilde{\gamma}_i^1)=e_i^2-\sigma_i e_0^2+E(\widetilde{T})\] 
\[\widetilde{\gamma}_i^2=e_i^2-\sigma_i e_0^2+E(\widetilde{T}), \tau_*(\widetilde{\gamma}_i^2)=e_i^3-\sigma_i e_0^3+E(\widetilde{T})\]
\[\widetilde{\gamma}_i^3=e_i^3+\sigma_i e_0^1+\sigma_i e_0^2+E(\widetilde{T}), \tau_*(\widetilde{\gamma}_i^1)=e_i^1+\sigma_i e_0^2+\sigma_i e_0^3+E(\widetilde{T})\]
Similarly,
\[p_*(\widetilde{\gamma}_i^1)=\gamma_i-\sigma_i \gamma_0\] 
\[p_*(\widetilde{\gamma}_i^2)=\gamma_i-\sigma_i \gamma_0\] 
\[p_*(\widetilde{\gamma}_i^3)=\gamma_i+2\sigma_i \gamma_0\] 

A basis for $\ker p_*$ is given by 
\[B^{\prime}_1=\{\widetilde{\gamma}_i^2-\widetilde{\gamma}_i^1\}_{e_i\in S\setminus\{e_0\}}\cup\{\widetilde{\gamma}_i^3-\widetilde{\gamma}_i^2-\sigma_i\widetilde{\gamma}_0\}_{e_i\in S\setminus\{e_0\}}\cup \{\widetilde{\gamma}_i^2-\widetilde{\gamma}_i^1\}_{e_i\notin S}\cup \{\widetilde{\gamma}_i^3-\widetilde{\gamma}_i^2\}_{e_i\notin S}\]
For the rest of our computation, we use the basis of $H_1(\Gamma, \Z)$ as given in \cite{LZ22}, namely
\[B^{\prime}=\{\gamma_0\}\cup\{\gamma_i-\sigma_i\gamma_0\}_{e_i\in S\setminus\{e_0\}}\cup\{\gamma_i\}_{e_i\notin S}\]
and therefore
\[B^{\prime}_2=p^*(B^{\prime})=\{\widetilde{\gamma}_0\}\cup\{\widetilde{\gamma}_i^1+\widetilde{\gamma}_i^2+\widetilde{\gamma}_i^3-\sigma_i\widetilde{\gamma}_0\}_{e_i\in S\setminus\{e_0\}}\cup\{\widetilde{\gamma}_i^1+\widetilde{\gamma}_i^2+\widetilde{\gamma}_i^3\}_{e_i\notin S}\]
Note that if $\Sigma=V/\Lambda$ is a $g$-dimensional pptav with $\{\lambda_1,\lambda_2,\ldots, \lambda_g\}$ a bsis for the lattice $\Lambda$, then we define
\[\vol^2(\Sigma)=\det(\lambda_i,\lambda_j)\] 
In particular $\vol^2(\jac(\Gamma))=\gram_{\Gamma}(B^{\prime}), \vol^2(\py(\widetilde{\Gamma}/\Gamma))=\gram_{P}(B^{\prime}_1)$. Here $\gram$ denotes the Gramian matrix, i.e., the matrix $G$ with entries $G_{ij}=(\lambda_i,\lambda_j)$.\\

For the basis $B^{\prime\prime}=B^{\prime}_1\cup B^{\prime}_2$ one sees that the change of basis matrix is block triangular with the top left entry 1 and a block $\begin{pmatrix}
1&1&1\\
-1&1&0\\
0&-1&1
\end{pmatrix}$ for each edge $e_i, i=1,2,\dots, g-1$. The determinant of this matrix is therefore $\pm 3^{g-1}$. It follows therefore that 
\[\vol^2(\jac(\widetilde{\Gamma}))=3^{2-2g}\gram_{\widetilde{\Gamma}}(B^{\prime\prime})\]
Now since for any $\gamma_1,\gamma_2\in H_1(\Gamma, \Z)$, it holds that $(p^*(\gamma_1), p^*(\gamma_2))_{\widetilde{\Gamma}}=3(\gamma_1, \gamma_2)_{\Gamma}$ we have that
\[\gram_{\widetilde{\Gamma}}(B^{\prime}_2)=3^g\gram_{\Gamma}(B^{\prime})\]
As is discussed in \cite{LZ22}, p. 23, for double covers of metric graphs we have $(\widetilde{\gamma}_1,\widetilde{\gamma}_2)_{\widetilde{\Gamma}}=0$ for every $\widetilde{\gamma}^{\prime}_1\in B^{\prime}_1$ and for every $\widetilde{\gamma}^{\prime}_2\in B^{\prime}_2$. This is due to the fact that $\iota_*(\widetilde{\gamma}^{\prime}_1)=\widetilde{\gamma}^{\prime}_1$ for every $\widetilde{\gamma}^{\prime}_1\in B^{\prime}_1$ and $\iota_*(\widetilde{\gamma}^{\prime}_2)=-\widetilde{\gamma}^{\prime}_2$ for every $\widetilde{\gamma}^{\prime}_2\in B^{\prime}_2$. For triple covers, the latter equalities do not hold true, however, a direct calculation shows that also in this case we have $(\widetilde{\gamma}_1,\widetilde{\gamma}_2)_{\widetilde{\Gamma}}=0$ from which it follows that
\[\gram_{\widetilde{\Gamma}}(B^{\prime\prime})=\gram_{\widetilde{\Gamma}}(B^{\prime}_1)\gram_{\widetilde{\Gamma}}(B^{\prime}_2)\]

Since $(\cdotp,\cdotp)_P=\frac{1}{3}(\cdotp,\cdotp)_{\widetilde{\Gamma}}$, it follows that 
\[\gram_{\widetilde{\Gamma}}(B^{\prime}_1)=3^{g-1}\gram_{P}(B^{\prime}_1)\]
Now by the above, we get
\[\frac{\vol^2(\jac(\widetilde{\Gamma}))}{3\vol^2(\jac(\widetilde{\Gamma}))}=\frac{3^{2-2g}\gram_{\widetilde{\Gamma}}(B^{\prime\prime})}{3\gram_{\Gamma}(B^{\prime})}=\frac{3^{1-2g}\gram_{\widetilde{\Gamma}}(B^{\prime}_1)\gram_{\widetilde{\Gamma}}(B^{\prime}_2)}{\gram_{\Gamma}(B^{\prime})}=\gram_{P}(B^{\prime}_1),\]
and this is equal to $\vol^2(\py(\widetilde{\Gamma}/\Gamma))$. We have therefore proved the following
\begin{proposition}\label{volumprym}
Let $\varphi:\widetilde{\Gamma}\to\Gamma$ be a free $\Z_3$-cover of metric graphs. Then we have the relation,
\[\vol^2(\py(\widetilde{\Gamma}/\Gamma))=\frac{\vol^2(\jac(\widetilde{\Gamma}))}{3\vol^2(\jac(\widetilde{\Gamma}))},\]
the volumes here are calculated using the intrinsic principal polarization. 
\end{proposition}

\begin{example}\label{triple exam}
Let us consider a free triple cover of a genus 2 graph as follows

\

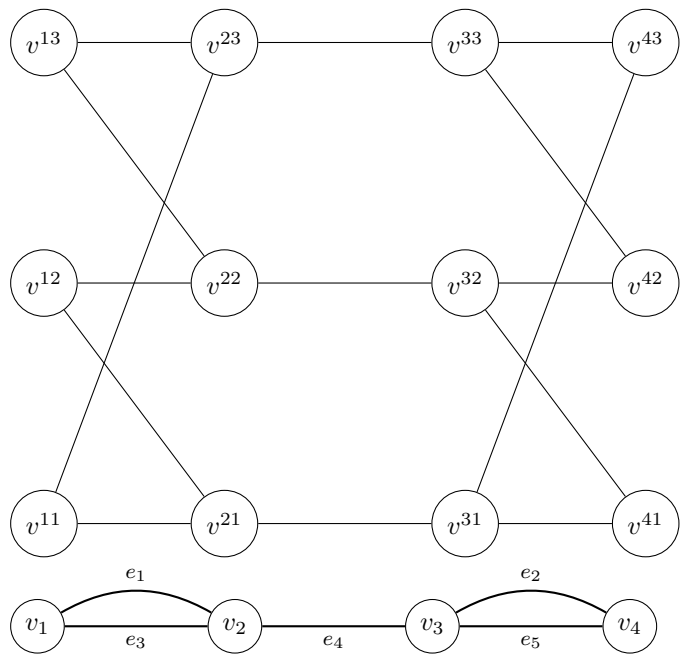
\begin{tikzpicture}
  [scale=.8,auto=left,every node/.style={circle,draw,fill=white}]

 \node (n9) at (2,13) {$v^{13}$};
  \node (n10) at (5,13)  {$v^{23}$};
  \node (n11) at (9,13)  {$v^{33}$};
 \node (n12) at (12,13) {$v^{43}$};

  \node (n6) at (2,9) {$v^{12}$};
  \node (n4) at (5,9)  {$v^{22}$};
  \node (n5) at (9,9)  {$v^{32}$};
 \node (n7) at (12,9) {$v^{42}$};

 \node (n1) at (12,5) {$v^{41}$};
  \node (n2) at (9,5)  {$v^{31}$};
  \node (n3) at (5,5)  {$v^{21}$};
 \node (n8) at (2,5)  {$v^{11}$};

  \foreach \from/\to in {n4/n5,n1/n2,n2/n3, n8/n3,n4/n6,n5/n7,n9/n10,n10/n11,n11/n12,n6/n3,n9/n4,n8/n10,n5/n1,n11/n7,n2/n12}
\draw (\from) -- (\to);

\end{tikzpicture}

\begin{tikzpicture} [scale=0.75]  

      \begin{scope}[>=latex,
                  		every node/.style={midway},
                       every edge/.style={draw=black,thick}]

      \begin{scope}[every node/.style={circle,draw,fill=white}]
         \node (1) at (0.5,4.5) {$v_1$};
         \node (2) at (4,4.5) {$v_2$};
         \node (4) at (7.5,4.5) {$v_3$};
         \node (5) at (11,4.5) {$v_4$};
      \end{scope}      

          \path [-] (2) edge[right] node [below]{\footnotesize $e_3$} (1);
          \path [-] (2) edge[right] node [below] {\footnotesize $e_4$} (4);
          \path [-] (5) edge[right] node [below]{\footnotesize $e_5$} (4);
\foreach \from/\to in {1/2}
\path (\from) edge [bend left] node [above] {\footnotesize $e_1$}(\to);
\foreach \from/\to in {4/5}
\path (\from) edge [bend left] node [above]  {\footnotesize $e_2$}(\to);
 \end{scope}
 \end{tikzpicture}
\end{example}
If we consider a model of the graphs where the length of each edge is equal to 1, then 
\[\vol^2(\jac(\widetilde{\Gamma}))=|\jac(\widetilde{\Gamma})|, \vol^2(\jac(\Gamma))=|\jac(\Gamma)|,\]
Further, we calculate that $|\jac(\widetilde{\Gamma})|=588$ and $|\jac(\Gamma)|=4$. Therefore we get
\[\vol^2(\py(\widetilde{\Gamma}/\Gamma))=\frac{\vol^2(\jac(\widetilde{\Gamma}))}{3\vol^2(\jac(\widetilde{\Gamma}))}=\frac{|\jac(\widetilde{\Gamma})|}{3|\jac(\Gamma)|}=49\]
In the next section we will calculate this volume using the Ihara Zeta function. \\

As the last result of this section, it's worth mentioning that for unramified tropical triple covers, Lemma \ref{dilcyc1} can be stated as follows:
\begin{lemma}\label{dilcyc trip}
Let $\varphi:\widetilde{\Gamma}\to\Gamma$ be an unramified $\Z_3$-cover of metric graphs with $\Gamma$ unaugmented, then
\begin{enumerate}
\item If $x\notin \gamma(\varphi),$ then $\varphi^{-1}(x)$ consists of at most $3$ points of genus 0.
\item If $x\in \gamma(\varphi),$ then $\varphi^{-1}(x)$ consists of points of genus\par
 $ \frac{1}{2}\deg_{\gamma_2(\varphi)}+\deg_{\gamma_3(\varphi)}-(d_{\widetilde{x}}-1)$,
\end{enumerate}
where $\deg_{\gamma_u(\varphi)}$ is defined in \ref{dilcyc} that.
\end{lemma}

\begin{remark}
Note that it follows directly from Lemma \ref{dilcyc trip} that $\deg_{\gamma_2(\varphi)}$ is even for every $\widetilde{x}$ and hence $\gamma_2(\varphi)$ is a cycle. 
\end{remark}

\subsection{ The Ihara zeta function and the Artin-Ihara $L$-function}

The Dedekind zeta function of a number field is defined for complex numbers $s$ with real part $Re(s) > 1$ by the Dirichlet series. The gtraph-theoretic counterpart for graphs $\Gamma$ using an Euler product over equivalence classes of  closed paths on $\Gamma$. Here we recall some elementary defintions and properties that we will use in this section. \\

Let $\Gamma$ be a graph with $n$ vertices and $m$ edges and genus $g$. A path $P$ of length $k$ is a sequence $P=e_1\cdots e_k$ of oriented edges of $\Gamma$ such that $t(e_i)=s(e_{i+1})$ for $i=1,\ldots, k-1$. The path $P$ is called \emph{closed} if $t(e_k)=s(e_{1})$ and \emph{reduced} if $e_{i+1}\neq\overline{e}_i$  for $i=1,\ldots, k-1$ and $e_{1}\neq\overline{e}_k$. The inetegr powers of closed paths are defined by concatenation, and a closed reduces path $P$ is called \emph{primitive} if it cannot be written in the form $P=Q^k$ for some closed path $Q$ and $k\geq 2$. Two reduced paths are considered equivalent if they differ by a choice of starting point; that is, we set $e_1\cdots e_k\equiv e_j\cdots e_k e_1\cdots e_k{j-1}$ for all $j=1,\ldots, k$. A \emph{prime} $P$ of $\Gamma$ is an equivalent class of primitve paths and has a well-defined length $\ell(P)$. The Ihara zeta function $\zeta(s,\Gamma)$ of a graph $\Gamma$ is the product 
\[\zeta(s,\Gamma)=\displaystyle \prod_p (1-s^{\ell(P)})^{-1}\]
We have the following three-term formula due to Bass 
\[\zeta(s,\Gamma)^{-1}= (1-s^2)^{g-1}\det(I_n-As+(Q-I_n)s^2)\]
This function vanishes at $s=1$ to order $g$ and it follows from the main result of \cite{north} that the leading Taylor coefficient computes the complexity; that is, the order of the Jacobian of $\Gamma$:
\[\zeta(s,\Gamma)^{-1}=(-1)^{g-1}2^g (g-1)|\jac(\Gamma)|(s-1)^g+O((s-1)^{g+1})\] 
Now let $\varphi:\widetilde{\Gamma}\to\Gamma$ be a free Galois covers of graphs with Galois group $G$. Let $\rho$ be a representation of $G$. Choose a prime $P$ of $\Gamma$ and let $v$ be its starting point and choose a vertex $\widetilde{v}\in\widetilde{\Gamma}$ lyign above $v$. The path $P$ lifts to a unique path $\widetilde{P}$ in $\widetilde{\Gamma}$ starting at $\widetilde{v}$ . The path $\widetilde{P}$ need not be cosed in general but its terminal vertex lies also over $v$ and thus there is a unique element $F(P, \widetilde{\Gamma}/\Gamma)\in G$ that maps $\widetilde{v}$ to the terminal vertex of $\widetilde{P}$. This element is called the \emph{Frobenius element} related to the path  and the \emph{Artin-Ihara $L$-function} is defined as
\[L(s,\rho, \widetilde{\Gamma}/\Gamma)=\displaystyle \prod_p \det(1-\rho(F(P, \widetilde{\Gamma}/\Gamma))s^{\ell(P)})^{-1}\]
Furthermore we have the \emph{Artinised valency} and \emph{Artinised adjacency} matrices as
\[Q_{\rho}=Q\otimes I_d, (A_{\rho})uv=\displaystyle\sum\rho(F(e)),\]
where the sum is taken over all edges between $u$ and $v$. We also have the following three-term determinant formula for the $L$-function
\[L(s,\rho, \widetilde{\Gamma}/\Gamma)^{-1}=(1-s^2)^{(g-1)d}\det(I_{nd}-A_{\rho} s+(Q_{\rho}-I_{nd})s^2\]
The zeta functions of $\widetilde{\Gamma}$ and $\Gamma$ are equal to the $L$-function evaluated respectively at the right regular and trivial representations $\rho_G$ and $1_G$:
\[\zeta(s,\widetilde{\Gamma})=L(s,\rho_G, \widetilde{\Gamma}/\Gamma), \zeta(s, \Gamma)=L(s,1_G, \widetilde{\Gamma}/\Gamma)\]
For a decomposable representation $\rho=\rho_1\oplus\rho_2$ we have
\[L(s,\rho, \widetilde{\Gamma}/\Gamma)=L(s,\rho_1, \widetilde{\Gamma}/\Gamma)L(s,\rho_2, \widetilde{\Gamma}/\Gamma).\]
Finally the relation between zeta functions of $\widetilde{\Gamma}$ and $\Gamma$ is given by
\[\zeta(s,\widetilde{\Gamma})=\zeta(s,\Gamma)\displaystyle \prod_p L(s,\rho, \widetilde{\Gamma}/\Gamma)^{d(\rho)},\]
where $d(\rho)$ denotes the degree of the irreducible representation $\rho$ and the product runs over all distinct  irreducible representations of $G$. 
In particular if $G=\Z_3$, let $\rho_1, \rho_2$ be the two 1-dimensional irreducible representations of $G$. We have
\[\zeta(s,\widetilde{\Gamma})^{-1}=\zeta(s, \Gamma)^{-1}L(s,\rho_1, \widetilde{\Gamma}/\Gamma)^{-1}L(s,\rho_2, \widetilde{\Gamma}/\Gamma)^{-1}\]
By the above we also have
\[\zeta(s,\widetilde{\Gamma})^{-1}=(-1)^{3g-3}2^{3g-2} (3g-3)|\jac(\widetilde{\Gamma})|(s-1)^{3g-2}+O((s-1)^{3g-1}),\]
\[\zeta(s, \Gamma)^{-1}=(-1)^{g-1}2^{g} (g-1)|\jac(\Gamma)|(s-1)^{g}+O((s-1)^{g+1}).\]
We also have
\[L(s,\rho_i, \widetilde{\Gamma}/\Gamma)^{-1}=(-1)^{g-1}2^{g-1}\det(Q_{\rho_i}-A_{\rho_i}) (s-1)^{g-1}+O((s-1)^{g}) \text{ for } i=1,2\]
Comparing the expansions of  $\frac{\zeta(s,\widetilde{\Gamma})^{-1}}{\zeta(s,\widetilde{\Gamma})^{-1}}$ and $L(s,\rho_i, \widetilde{\Gamma}/\Gamma)$ we see that
\[\frac{|\jac(\widetilde{\Gamma})|}{3|\jac(\Gamma)|}=\frac{1}{9}\det(Q_{\rho_1}-A_{\rho_1})\det(Q_{\rho_2}-A_{\rho_2})\]
Let us revisit the tripe cover in Example \ref{triple exam}. In this case, since the representations $\rho_1, \rho_2$ are 1-dimensional, $Q_{\rho_i}=Q$ for $i=1,2$ and we have
\[Q_{\rho_1}-A_{\rho_1}=\begin{pmatrix}
2&\xi&0&0\\
\xi^2&3&-1&0\\
0&-1&3&\xi\\
0&0&\xi^2&2
\end{pmatrix}\]
and
\[Q_{\rho_2}-A_{\rho_2}=\begin{pmatrix}
2&\xi^2&0&0\\
\xi&3&-1&0\\
0&-1&3&\xi^2\\
0&0&\xi&2
\end{pmatrix},\]
where $\xi$ denotes a primitive 3rd root of unity. We  compute $\det(Q_{\rho_i}-A_{\rho_i})=21$ and therefore
\[\frac{|\jac(\widetilde{\Gamma})|}{3|\jac(\Gamma)|}=\frac{1}{9}\det(Q_{\rho_1}-A_{\rho_1})\det(Q_{\rho_2}-A_{\rho_2})=49.\]

\end{document}